\documentclass{amsart}

\usepackage{graphicx} 
\usepackage{amsmath}						
\usepackage{amssymb}
\usepackage{amsthm}
\usepackage{amscd}
\usepackage{amsfonts}
\usepackage{mathtools}
\usepackage{color}
\usepackage{xcolor}
\usepackage[normalem]{ulem}
\usepackage[colorlinks=true,linkcolor=blue,citecolor=blue,urlcolor=blue]{hyperref}
\usepackage{cleveref}

\theoremstyle{plain}
	\newtheorem{theorem}{Theorem}[section]
	\newtheorem{lemma}[theorem]{Lemma}
    \newtheorem{cor}[theorem]{Corollary}
    \newtheorem{prop}[theorem]{Proposition}

\theoremstyle{definition}
    \newtheorem{defn}[theorem]{Definition}
    \newtheorem*{defn*}{Definition}
    \newtheorem{example}[theorem]{Example}
    \newtheorem*{example*}{Example}
    
\theoremstyle{remark}
	\newtheorem{remark}[theorem]{Remark}
	\newtheorem*{remark*}{Remark}
    \newtheorem{question}[theorem]{Question}
    \newtheorem*{question*}{Question}

\newcommand{\bbF}{\mathbb{F}}
\newcommand{\bbP}{\mathbb{P}}

\newcommand{\mcF}{\mathcal{F}}

\newcommand{\TC}{\operatorname{TC}}
\newcommand{\PTC}{\bbP\operatorname{TC}}
\newcommand{\CTC}{\overline{\operatorname{TC}}}

\usepackage{todonotes}

\title{Obstructions to embedding singular curves in toric varieties}

\author{Maya Banks} 
\address{Department of Mathematics, Statistics, and Computer Science
(M/C 249), University of Illinois Chicago,
851 S. Morgan Street, Chicago, IL 60607-7045, USA}
\email{mayadb@uic.edu}
\author{Izzet Coskun} 
\address{Department of Mathematics, Statistics, and Computer Science
(M/C 249), University of Illinois Chicago,
851 S. Morgan Street, Chicago, IL 60607-7045, USA}
\email{icoskun@uic.edu}
\author{Kevin Tucker}
\address{Department of Mathematics, Statistics, and Computer Science
(M/C 249), University of Illinois Chicago,
851 S. Morgan Street, Chicago, IL 60607-7045, USA}
\email{kftucker@uic.edu}
\thanks{Banks was partially supported by the National Science Foundation grant DMS \#2037569. Coskun was partially supported by the National Science Foundation grant DMS \#2200684 and Simons Foundation Travel Support for Mathematicians SFI-MPS-TSM-00013580. Tucker was partially supported by
National Science Foundation grant \#2501904 and Simons Foundation Travel Support for Mathematicians SFI-MPS-TSM-00014083.}

\date{}

\begin{document}

\begin{abstract}
For every integer $d \geq 2$, we show that there exists an irreducible, reduced curve which embeds in $\bbP^d$ but in no projective normal toric variety of dimension less than $d$. In particular, there exist reduced irreducible curves that embed in $\bbP^d$ but do not embed in any weighted projective space of dimension less than $d$.
\end{abstract}

\maketitle

\section{Introduction}

Every smooth projective curve admits an embedding into $\bbP^3$. 
For singular curves, however, there is a local obstruction: if $p\in C$ is a singular point with Zariski tangent space $T_pC$ of dimension $d$, then $C$ cannot embed into any smooth projective variety of dimension strictly smaller than $d$.
A natural way to circumvent this obstruction is to allow the ambient variety itself to be singular. For a fixed curve $C$ with a singular point $p$, one may hope to embed $C$ into a  threefold $X$ possessing a singular point whose tangent space is sufficiently large to contain $T_pC$. This leads to the following question.

\begin{question}\label{q:wps}
Can every projective curve be embedded into some weighted projective threefold?
\end{question}

Question \ref{q:wps} naturally fits into the broader program of understanding algebraic varieties through the geometry of their subvarieties. For instance, rational and elliptic curves often play a central role in studying the birational geometry and arithmetic of varieties. In characteristic zero, Sankaran \cite{San11} proved that every smooth projective curve embeds in any smooth projective rationally connected variety of dimension at least 3.  Conversely, Lou \cite{Lou24} showed that, in characteristic zero, a smooth projective variety that admits an embedding of every smooth projective curve  contains a rationally connected subvariety of dimension at least 3.

In Question \ref{q:wps}, some restrictions are clearly necessary for nonreduced curves. For instance, a nonreduced curve $C$ whose reduced structure is irreducible with positive genus and whose tangent spaces all have dimension greater than $3$ cannot embed in a weighted projective threefold: tangent-space injectivity forces the support of the image of $C$ to lie in the singular locus of the weighted projective threefold, which is a union of points and rational curves. We therefore restrict our attention to reduced, irreducible projective curves.

Throughout the paper, we work over an uncountable algebraically closed base field $k$. Our main theorem is the following.

\begin{theorem}\label{main_thm}
    For every integer $d\geq 2$, there exists a reduced and irreducible projective curve $C_d \subset \bbP^d$ such that $C_d$ does not embed in any projective normal toric variety of dimension less than $d$.
\end{theorem}

Since weighted projective spaces are normal toric varieties, we obtain as an immediate corollary that, for every $d\geq2$, there exists a reduced and irreducible projective curve $C_d\subset \bbP^d$ which does not embed in any weighted projective space of dimension less than $d$. Taking $d=4$ gives a negative answer to Question \ref{q:wps}.


Our approach to Theorem \ref{main_thm} is to study the problem locally by identifying obstructions to embedding certain arrangements of concurrent lines. More precisely, we analyze unions of lines through a common point in $\mathbb{P}^d$ and show that a sufficiently general arrangement of this kind does not admit a vertex-preserving closed immersion into the completed tangent cone at any point in a projective normal toric variety of dimension $\delta<d$, provided the number of lines is large enough. We then construct curves with a unique singular point $p$ whose completed tangent cone at $p$ contains a prescribed union of concurrent lines. 

In fact, we prove a more general avoidance statement: if $\mcF_i\to B_i$ is a countable collection of families whose members have dimension strictly smaller than $d$, with $\dim B_i\leq M$ for a constant $M$ independent of $i$, then there exists a reduced, irreducible curve $C\subset \bbP^d$ which cannot embed in any member of any of these families; see Theorem~\ref{main-thm-count-fam} for a precise statement.

The local obstruction used in Theorem \ref{main_thm} is sensitive to the normal toric hypothesis. In contrast, if one allows linear projections of weighted projective planes, then every reduced union of concurrent lines admits such an embedding (see Theorem \ref{nonnormal}). More generally, we show that these line configurations always admit an embedding into a nonnormal toric threefold; see Theorem~\ref{thm:equivariant-pinching}.

We will see, however, that eliminating the local obstruction is not sufficient. Even for smooth curves there exist subtler global obstructions to embedding into a fixed projective variety.
For example, general smooth curves of genus \(g>10\) do not embed into cones over smooth rational curves; see Remark \ref{rem:global-obstruction}.
From this perspective, it would be interesting to determine the smallest natural class of projective threefolds into which every reduced irreducible projective curve can be embedded. For example, every curve can be embedded in a complete intersection surface or threefold.

\subsection{Acknowledgments} We would like to thank Nolan Schock for discussions related to this project. We thank the anonymous referee for invaluable suggestions. 

\section{Preliminaries}
In this section, we fix notation and collect some standard facts concerning tangent cones, toric varieties, weighted projective spaces and Hirzebruch surfaces. We continue to write $k$ for our uncountable algebraically closed base field. All schemes are $k$-schemes unless otherwise specified, and by a variety we mean an integral separated scheme of finite type over $k$.

 \subsection{Tangent cones}\label{subsec:tangent-cones} Let $X$ be a finite-type $k$-scheme and let $p\in X$ be a closed point with maximal ideal $\mathfrak m_p\subset \mathcal O_{X,p}$. Set $G_p(X)=\operatorname{gr}_{\mathfrak m_p}\mathcal O_{X,p}$.
We write
\[
    \TC_pX=\operatorname{Spec} G_p(X),\qquad
    \PTC_pX=\operatorname{Proj} G_p(X),
\]
for the tangent cone of $X$ at $p$ and its projectivization. We also write
\[
    \CTC_pX=\operatorname{Proj}(G_p(X)[z]),
\]
 for the projective completion of the tangent cone, where $z$ has degree one. We call $\CTC_pX$ the completed tangent cone. It has a distinguished vertex, which we identify with $p$, corresponding on the affine chart $D_+(z)\cong \TC_pX$ to the homogeneous maximal ideal $G_p(X)_+$. Then $\TC_pX$ is the open set $D_+(z)\subset \CTC_pX$, while $\PTC_pX$ is the hyperplane at infinity $V(z)\subset \CTC_pX$.

If $G$ is a standard graded $k$-algebra and $T=\operatorname{Proj}(G[z])$, then the completed tangent cone of $T$ at its vertex is canonically isomorphic to $T$, since the local ring at the vertex is $G_{G_+}$ and its associated graded ring with respect to the homogeneous maximal ideal is canonically $G$.

Even when $X$ is integral, $\TC_pX$ need not be reduced or irreducible. If $X$ is pure of dimension $\delta$, then every irreducible component of $\TC_pX$ has dimension $\delta$ \cite[Appendix B.6.6]{F98}, and hence $\PTC_pX$ has dimension $\delta-1$ when $\delta>0$. We will use the standard fact that if $Y\hookrightarrow X$ is a closed immersion and $y\in Y$ maps to $x\in X$, then the induced degree-preserving surjection on associated graded rings gives closed immersions
\[
    \TC_yY \hookrightarrow \TC_xX,\qquad
    \PTC_yY \hookrightarrow \PTC_xX,\qquad
    \CTC_yY \hookrightarrow \CTC_xX.
\]
These facts follow from the construction of tangent cones via associated graded rings; see \cite[Appendix B]{F98}.

\subsection{Toric varieties and weighted projective spaces} A (not necessarily normal) toric variety is a variety $X$ that contains an algebraic torus as a dense open subset such that the natural action of the torus on itself extends to $X$. Over the algebraically closed field $k$, there are countably many normal toric varieties of fixed dimension up to isomorphism. Indeed, a normal toric variety of dimension $d$ is determined by a finite rational fan in a lattice isomorphic to $\mathbb{Z}^d$, and there are only countably many such fans. Moreover, each toric variety has only finitely many tangent cones up to isomorphism: a toric variety has finitely many torus orbits, and the torus action identifies the local rings, hence also the tangent cones, at any two points in the same orbit. For background on toric varieties, see \cite{CLS11}.

Let $a_0 \leq \cdots \leq a_n$ be a sequence of positive integers such that the greatest common divisor of any $n$ of the integers is 1. Let $\mathbb{P}(a_0, \dots, a_n)$ denote the \emph{weighted projective space} which is the quotient $(\mathbb{A}^{n+1} \setminus \{0\}) / \mathbb{G}_m$, where the multiplicative group $\mathbb{G}_m$ acts by \[\lambda \cdot(x_0, x_1, \dots, x_n) = (\lambda^{a_0} x_0, \lambda^{a_1} x_1, \dots, \lambda^{a_n} x_n).\] Equivalently, $\mathbb{P}(a_0,\ldots,a_n)=\operatorname{Proj}(k[x_0,\ldots,x_n])$, where $\deg x_i=a_i$. Weighted projective spaces are normal projective toric varieties; see, for instance, \cite{D82}.



\subsection{Hirzebruch surfaces}\label{subsec:hirzebruch-surfaces}
Finally, for $n\geq1$, write
$\bbP(1,1,n)=\operatorname{Proj}k[s,t,u]$, where
$\deg s=\deg t=1$ and $\deg u=n$. The degree-$n$ monomials define a
closed immersion
\[
\bbP(1,1,n)\hookrightarrow\bbP^{n+1},\qquad
[s:t:u]\longmapsto
[u:s^n:s^{n-1}t:\cdots:t^n].
\]
Its image is the cone over the rational normal curve of degree $n$.

Consider
the Hirzebruch surface
\[
\bbF_n\cong
\bbP\left(
\mathcal O_{\bbP^1}\oplus\mathcal O_{\bbP^1}(n)
\right).
\]
We denote by $E$ the negative section, with $E^2=-n$, and by $F$ the
class of a fiber of the ruling $\bbF_n\to\bbP^1$. The complete linear
series $|E+nF|$ defines a birational morphism
$\rho:\bbF_n\to\bbP(1,1,n)$ which contracts $E$ to the vertex of the
cone and maps the fibers of the ruling to the lines through the
vertex. When $n=1$, we have
$\bbP(1,1,1)\cong\bbP^2$, and $\rho:\bbF_1\to\bbP^2$ is the blowdown
of the $(-1)$-section $E$ to a smooth point. For $n\geq2$, the morphism $\rho$ is the minimal resolution of
$\bbP(1,1,n)$. For the basic geometry of Hirzebruch surfaces, see
\cite[Chapter V, Section 2]{Har77}.

More generally, let $R\subset\bbP^r$ be a smooth rational curve of
degree $n$, and let $\hat R$ be its projective cone. Since
$R\simeq\bbP^1$ and
$\mathcal O_R(1)\simeq\mathcal O_{\bbP^1}(n)$, the integral closure of
the homogeneous coordinate ring of $R$ is the homogeneous coordinate ring of
the rational normal curve of degree $n$. Hence the normalization map
of $\hat R$ is
$\nu:\bbP(1,1,n)\to\hat R$. Since $R$ is smooth, $\nu$ is an
isomorphism away from the cone point, and therefore
$\phi=\nu\circ\rho:\bbF_n\to\hat R$ restricts to an isomorphism
$\bbF_n\setminus E\simeq\hat R\setminus\{\text{cone point}\}$.

\section{Line arrangements}
In this section, we discuss embeddings of concurrent lines in  toric varieties.

\begin{defn}\label{defcgamma}
 Let $d \geq 2$, let $H$ be a hyperplane in $\bbP^d$, and let $\Gamma \subseteq H$ be a set of $m$ points spanning $H$.  In particular, $m\geq d$. Let $p\in \bbP^d$ be a point not contained in $H$ and let $C_{\Gamma}$ denote the cone over $\Gamma$ with vertex $p$. 
\end{defn}

Observe that $C_{\Gamma}$ consists of $m$ concurrent lines through $p$, and its intersection with $H$ is $\Gamma$. Under the natural identification of $H$ with $\bbP \operatorname{T}_p\bbP^d$, the spanning assumption also gives $H=\bbP \operatorname{T}_pC_\Gamma$, and we have $\PTC_p(C_\Gamma)=\Gamma$ and $C_\Gamma\cong \CTC_p(C_\Gamma)$. We denote the cardinality of $\Gamma$ by $|\Gamma| = m$.

\begin{lemma}\label{lem-isom}
    Two curves $C_{\Gamma}$ and $C_{\Gamma'}$ are isomorphic if and only if $\Gamma$ and $\Gamma'$ are projectively equivalent.
\end{lemma}

\begin{proof}
    If $\Gamma$ and $\Gamma'$ are projectively equivalent, then $C_{\Gamma}$ and $C_{\Gamma'}$ are projectively equivalent, hence isomorphic. 
    
    Conversely, suppose $f:C_{\Gamma}\to C_{\Gamma'}$ is an isomorphism. The vertex $p$ of $C_{\Gamma}$ is its unique singular point, and the same holds for the vertex $p'$ of $C_{\Gamma'}$. Hence $f(p)=p'$, and $df$ at $p$ induces a projective linear isomorphism $\bbP \operatorname{T}_pC_\Gamma\cong\bbP \operatorname{T}_{p'}C_{\Gamma'}$ carrying $\PTC_p(C_\Gamma)=\Gamma$ onto $\PTC_{p'}(C_{\Gamma'})=\Gamma'$. Thus $\Gamma$ and $\Gamma'$ are projectively equivalent.
\end{proof}

\begin{cor}
    \label{cor:dimofCgammas}
    When $m > d$, the moduli space of unordered configurations of $m$ distinct points spanning $\bbP^{d-1}$ up to projective equivalence has dimension $m(d-1)-d^2+1$. Consequently, the curves $C_\Gamma$ depend on $m(d-1)-d^2+1$ parameters up to isomorphism.
\end{cor}

\begin{proof}
    The locus of $m$ unordered distinct points in $\bbP^{d-1}$ that span is a dense open subset $U$ of the symmetric product $(\bbP^{d-1})^{(m)}$, which has dimension $m(d-1)$. 
    Since $\operatorname{PGL}(d)$ has dimension $d^2-1$ and a general configuration has finite stabilizer when $m>d$, 
    the rational quotient $U/\operatorname{PGL}(d)$ has dimension
    $m(d-1)-(d^2-1)=m(d-1)-d^2+1$.
    The final assertion follows from Lemma~\ref{lem-isom}.
\end{proof}

\begin{example}
     Assume $d\geq 3$. The minimal embedding of $\bbP(1,1,d-1)$ is the cone in $\bbP^d$ over a rational normal curve of degree $d-1$ in $\bbP^{d-1}$. Recall that through any $d+2$ linearly general points in $\bbP^{d-1}$, there exists a unique rational normal curve of degree $d-1$. Hence, if $\Gamma \subset \bbP^{d-1}$ is a set of linearly general points with $|\Gamma| \leq d+2$, then $C_{\Gamma}$ embeds in $\bbP(1,1,d-1)$. 

     On the other hand, if $|\Gamma| > d+2$, then $\Gamma$ does not in general lie on a rational normal curve. We conclude that the general $C_{\Gamma}$ embeds in $\bbP(1,1,d-1)$ if and only if $|\Gamma| \leq d+2$.
\end{example}

\begin{defn}
    By a \emph{family of $\delta$-dimensional varieties}, we mean a flat morphism $\mcF\to B$ of quasi-projective schemes of finite type over $k$, where $B$ is integral and every geometric fiber is an integral scheme of dimension $\delta$.
\end{defn}

Given a family of $\delta$-dimensional varieties $\mcF\to B$, the
tangent cones to its fibers vary in finitely many flat families whose
bases have dimension at most
$\dim\mcF=\dim B+\delta$. To see this, consider
$\mcF\times_B\mcF$: the first factor records a point $q$ in a fiber
$X$, while the second records a varying point of the same fiber. The
normal cone to the diagonal packages the tangent cones $\TC_qX$
\cite[\href{https://stacks.math.columbia.edu/tag/0630}{Definition 0630},
\href{https://stacks.math.columbia.edu/tag/0636}{Definition 0636}]
{stacks-project}. Although formation of the normal cone need not
commute with specialization, generic flatness and Noetherian induction
give a finite locally closed stratification of $\mcF$ such that, on each stratum, formation of the normal cone commutes with specialization and the resulting family is flat
\cite[\href{https://stacks.math.columbia.edu/tag/0H3Y}{Section 0H3Y}]
{stacks-project}. Taking relative Proj, before and after adjoining a
degree-one variable, then gives flat projective families whose fibers
are $\PTC_qX$ and $\CTC_qX$, respectively
\cite[\href{https://stacks.math.columbia.edu/tag/0B3U}{Lemma 0B3U},
\href{https://stacks.math.columbia.edu/tag/0D4C}{Lemma 0D4C}]
{stacks-project}.

\begin{theorem}\label{main-thm-family-local}
    Let $\mcF\to B$ be a family of $\delta$-dimensional varieties with $\delta < d$, and suppose $\dim B \leq M$. {If $|\Gamma|=m$ and $m(d-\delta)>M+\delta+d^2-1$, then the general $C_\Gamma$ does not admit a vertex-preserving closed immersion $C_\Gamma\cong \CTC_p(C_\Gamma)\hookrightarrow \CTC_qX$} for any point $q$ of any member $X$ of $\mcF$. In particular, the general $C_\Gamma$ does not embed into any member of $\mcF$.
\end{theorem}
\begin{proof}
    {We may assume $\delta>0$. Consider a vertex-preserving closed immersion $C_\Gamma\hookrightarrow \CTC_qX$ for some point $q$ of a member $X$ of $\mcF$. Passing to completed tangent cones at the vertices gives a $\mathbb G_m$-equivariant closed immersion $C_\Gamma\cong\CTC_p(C_\Gamma)\hookrightarrow \CTC_v(\CTC_qX)$, where $v$ is the vertex of $\CTC_qX$. By the grading-compatible identification $\CTC_v(\CTC_qX)\cong \CTC_qX$ from Section~\ref{subsec:tangent-cones}, we may regard this as a $\mathbb G_m$-equivariant closed immersion $C_\Gamma\hookrightarrow\CTC_qX$. Restricting to the hyperplanes at infinity gives a closed immersion $\Gamma=\PTC_p(C_\Gamma)\hookrightarrow \PTC_qX$. Since the original immersion is injective on Zariski tangent spaces at the vertices, the images of the points of $\Gamma$ span a $(d-1)$-plane. As $\dim \PTC_qX=\delta-1$, it follows that the dimension of the space of possible configurations $\Gamma$ in $\PTC_qX$ is at most $m(\delta - 1)$. From the discussion above, the dimensions of the families of isomorphism classes of tangent cones occurring in $\mcF$ are bounded above by $M + \delta$. 
    Thus the configurations $\Gamma$ for which $C_\Gamma$ admits such an embedding into $\CTC_qX$, for some point $q$ of some member $X$ of $\mcF$, vary up to projective equivalence in a family of dimension at most $M+\delta+m(\delta-1)$.
    If $m(d-\delta)>M+\delta + d^2-1$, then
    $M + \delta + m(\delta - 1)$ is strictly smaller than 
    the dimension of the moduli space of point configurations
    from Corollary~\ref{cor:dimofCgammas}. 
    Hence their projective equivalence classes lie in a proper subvariety of that moduli space.
    }
    Finally, if $C_\Gamma$ admitted an embedding into a member $X$ of $\mcF$, with $p$ mapping to $q$, functoriality of completed tangent cones would give $C_\Gamma\cong \CTC_p(C_\Gamma)\hookrightarrow \CTC_qX$, contradicting the preceding conclusion.
\end{proof}

\begin{cor}\label{cor-count-fam}
    Let $\mcF_i \to B_i$ be a countable collection of families of varieties of dimensions strictly smaller than $d$, and suppose $\dim(B_i) \leq M$ for every $i$. If $|\Gamma|=m\geq M+d^2+d-1$, then the very general $C_\Gamma$ does not admit a vertex-preserving closed immersion $C_\Gamma\cong \CTC_p(C_\Gamma)\hookrightarrow \CTC_qX$ for any point $q$ of any member $X$ of any $\mcF_i$. In particular, the very general $C_\Gamma$ does not embed into any member of any $\mcF_i$.
\end{cor}
\begin{proof}
    Let $\delta_i<d$ be the dimension of the members of $\mcF_i$. Since $m\geq M+d^2+d-1$, we have $m(d-\delta_i)>M+\delta_i+d^2-1$ for every $i$. 
    By \autoref{main-thm-family-local}, for each $i$, the projective equivalence classes of configurations $\Gamma$ for which $C_\Gamma$ admits such an embedding into $\CTC_qX$ for some point $q$ of some member $X$ of $\mcF_i$ lie in a proper subvariety of the moduli space of point configurations. Since there are countably many families and the base field $k$ is uncountable, the projective equivalence class of a very general $\Gamma$ avoids all of these proper subvarieties.
\end{proof}

\begin{cor}
  {If $|\Gamma|=m>d^2-1$, then the very general $C_\Gamma$ does not admit a vertex-preserving closed immersion $C_\Gamma\cong\CTC_p(C_\Gamma)\hookrightarrow\CTC_qX$ for any point $q$ of any projective normal toric variety $X$ of dimension less than $d$. In particular, the very general $C_\Gamma$ does not embed into any projective normal toric variety of dimension less than $d$.}
\end{cor}
\begin{proof}
    {There are only countably many normal toric varieties of dimension less than $d$, and each has only finitely many tangent cones up to isomorphism. Thus the completed tangent cones that arise form a countable collection of fixed cones. Fix one of them, call it $\overline T$, and write $\delta=\dim \overline T<d$. 
    As argued at the beginning of the proof of \autoref{main-thm-family-local}, the locus of projective equivalence classes of configurations $\Gamma$ for which $C_\Gamma$ admits a vertex-preserving closed immersion into $\overline{T}$ has dimension at most $m(\delta - 1)$.
    Since $m>d^2-1$ and $\delta<d$, we have $m(\delta-1)<m(d-1)-d^2+1$, so this locus is contained in a proper subvariety of 
    the moduli space of point configurations
    from Corollary \ref{cor:dimofCgammas}. Since there are only countably many such cones and the base field $k$ is uncountable, the very general $C_\Gamma$ admits no such immersion. Finally, an embedding into a projective normal toric variety $X$ would induce a closed immersion $C_\Gamma\cong\CTC_p(C_\Gamma)\hookrightarrow\CTC_qX$ at the image $q$ of $p$.}
\end{proof}

The normal toric hypothesis is essential in this local obstruction, as the next two results will demonstrate. We first show that $C_{\Gamma}$ can always be embedded in a linear projection of a weighted projective plane $\bbP(1,1,n)$ if $n$ is sufficiently large. However, these projections need not be equivariant, so the resulting surfaces are not in general toric. Nevertheless, we then show there is also a genuinely toric nonnormal version: every $C_\Gamma$ embeds into a nonnormal projective toric threefold.

\begin{theorem}\label{nonnormal}
   The curve $C_{\Gamma}$ can be embedded in a linear projection of the weighted projective plane $\bbP(1,1,n)$ if $n$ is sufficiently large. 
\end{theorem}

\begin{proof}
   
   Let $\Gamma\subset H\simeq\bbP^{d-1}$. 
   By Lagrange interpolation, there exists $n\gg 0$ such that there is a basepoint-free linear subsystem of $|\mathcal O_{\bbP^1}(n)|$ defining a morphism $f:\bbP^1\to H$ whose image $R=f(\bbP^1)$ contains $\Gamma$. If $d\geq4$, one can moreover choose the subsystem so that $R$ is smooth. Let $\hat R\subset\bbP^d$ be the cone over $R$ with vertex $p$. Since $\Gamma\subset R$, the cone $\hat R$ contains $C_\Gamma$.

   The complete linear system $|\mathcal O_{\bbP^1}(n)|$ gives the rational normal curve of degree $n$, and the subsystem defining $f$ realizes $R$ as a linear projection of that rational normal curve. Taking cones, $\hat R$ is therefore the corresponding
    linear projection of the standard embedding of $\bbP(1,1,n)$
    as described in Section~\ref{subsec:hirzebruch-surfaces}.
\end{proof}

\begin{remark}
For a general choice of the interpolating curve $R$, the surface $\hat R$ is not toric. Indeed, assume for simplicity that $d \geq 4$ so that $f$ is an embedding. Being toric would force the dense torus to fix the vertex and hence act on $R=\PTC_p\hat R$. This action cannot be trivial, since otherwise the torus orbits off the vertex would lie in the ruling lines of the cone. Thus, after identifying $R\simeq\bbP^1$, the linear series defining $R\subset H$ would be spanned by $\mathbb G_m$-eigenvectors, so $R$ would be projectively equivalent to a monomial curve. A general interpolating $R$ is not projectively equivalent to such a curve.
\end{remark}

\begin{theorem}\label{thm:equivariant-pinching}
 There exist a projective nonnormal toric threefold $X_\Gamma$ and 
 a closed immersion $C_\Gamma\hookrightarrow X_\Gamma$ taking the vertex $p$ of $C_\Gamma$ to a torus-fixed point $x\in X_\Gamma$.
\end{theorem}

\begin{proof}
We first describe the construction informally. The normalization of $C_\Gamma$ is a disjoint union of copies of $\bbP^1$. We place these copies as disjoint torus-invariant curves in a smooth projective toric threefold and then glue them together using a pushout construction to obtain $C_\Gamma$ as a closed subscheme. Since the construction is equivariant and does not change the dense torus, the resulting threefold is toric.

Put $m=|\Gamma|$, and let
$\eta:\widetilde C_\Gamma=\coprod_{i=1}^m\bbP^1\to C_\Gamma$
be the normalization, with $\eta^{-1}(p)=\{p_1,\ldots,p_m\}$. Choose a smooth projective toric surface $S$ with distinct torus-fixed points $s_1,\ldots,s_m$, and set $Y=\bbP^1\times S$. Then $Y$ is a smooth projective toric threefold with torus $T=\mathbb G_m\times T_S$. Identify the $i$-th component of $\widetilde C_\Gamma$ equivariantly with
$F_i=\bbP^1\times\{s_i\}$, taking $p_i$ to $y_i=(0,s_i)$, and write $F=\coprod_iF_i$. Let $T$ act on $C_\Gamma$ through the projection $T\to\mathbb G_m$ and the scaling action on $C_\Gamma$. Then $\eta:F\to C_\Gamma$ is $T$-equivariant.

Since $F\hookrightarrow Y$ is a closed immersion, $\eta$ is finite, and $Y$ and $C_\Gamma$ are projective (so that all finite sets of points are contained in an affine open), \cite[Theorem 5.4]{F03} (see also \cite[\href{https://stacks.math.columbia.edu/tag/0ECH}{Section 0ECH}]{stacks-project}) gives the pushout
$X_\Gamma=Y\amalg_F C_\Gamma$,
together with a finite morphism $\nu:Y\to X_\Gamma$ and a closed immersion $C_\Gamma\hookrightarrow X_\Gamma$. Let $x$ be the image of the vertex $p$. Since $\eta$ is surjective, $\nu$ is surjective. Moreover, $\eta$ restricts to an isomorphism
$F\setminus\{y_1,\ldots,y_m\}\simeq C_\Gamma\setminus\{p\}$.
The affine-local description of the pushout in
\cite[\href{https://stacks.math.columbia.edu/tag/0E25}{Proposition 0E25}]{stacks-project}
therefore shows that $\nu$ induces an isomorphism
$Y\setminus\{y_1,\ldots,y_m\}\simeq X_\Gamma\setminus\{x\}$.
The scheme $X_\Gamma$ is separated and locally of finite type by
\cite[\href{https://stacks.math.columbia.edu/tag/0E26}{Lemma 0E26} and
\href{https://stacks.math.columbia.edu/tag/0E27}{Lemma 0E27}]{stacks-project}.
Since $\nu$ is surjective and $Y$ is quasi-compact, $X_\Gamma$ is also
quasi-compact, and hence is of finite type.

Since $C_\Gamma$ is reduced, the natural map
$\mathcal O_{C_\Gamma}\to\eta_*\mathcal O_F$
is injective. As $Y$ is integral, \cite[Proposition 5.6(2)]{F03} implies that $X_\Gamma$ is integral. The isomorphism above shows that \(\nu\) is birational, so \(X_\Gamma\) is a threefold. Since  $\nu$ is also finite and $Y$ is normal, $\nu$ is the normalization of $X_\Gamma$. Since $m\geq d\geq2$, the distinct points $y_1,\ldots,y_m$ all map to $x$, so $\nu$ is not an isomorphism and hence $X_\Gamma$ is nonnormal.
The compatible $T$-actions descend to $X_\Gamma$ by flat base change for pushouts \cite[Lemma 4.4]{F03}. The dense torus of $Y$ is disjoint from $F$ and is unchanged by the construction. Thus $X_\Gamma$ is toric, and $x$ is $T$-fixed.

It remains to prove that $X_\Gamma$ is projective. Since $Y$ is proper and $\nu:Y\to X_\Gamma$ is surjective, \cite[\href{https://stacks.math.columbia.edu/tag/03GN}{Lemma 03GN}]{stacks-project} implies that $X_\Gamma$ is proper. Choose a hyperplane section $D\subseteq Y$ avoiding $y_1,\ldots,y_m$. Since $D$ is disjoint from the non-isomorphism locus of $\nu$, its image $D_X=\nu(D)$ is an effective Cartier divisor on $X_\Gamma$ and
$\nu^*\mathcal O_{X_\Gamma}(D_X)\simeq\mathcal O_Y(D)$.
By \cite[\href{https://stacks.math.columbia.edu/tag/0B5V}{Lemma 0B5V}]{stacks-project}, $\mathcal O_{X_\Gamma}(D_X)$ is ample. Therefore $X_\Gamma$ is projective.
\end{proof}

\begin{remark}
Note that the point $x$ has no $T$-invariant affine open neighborhood. Indeed, if $U\subset X_\Gamma$ were such a neighborhood, then, since $\nu$ is finite and $T$-equivariant, $\nu^{-1}(U)$ would be a $T$-invariant affine open subset of $Y$ containing the distinct torus-fixed points $y_1,\ldots,y_m$. This is impossible, since an affine toric variety has at most one torus-fixed point \cite[Proposition 1.3.2]{CLS11}. Thus $X_\Gamma$ is toric under the convention used here, but not in the sense of \cite[Definitions 42 and 43]{GPT14}, where a finite cover by $T$-invariant affine open subsets is required.
\end{remark}

\section{Main result}
Recall from Definition \ref{defcgamma} that $C_{\Gamma}$ is a union of $| \Gamma | = m$ concurrent lines in $\bbP^d$. In this section, given $C_{\Gamma}$, we construct a reduced, irreducible curve with a single singular point $p$ whose completed tangent cone at $p$ contains $C_{\Gamma}$. This will allow us to prove our main theorem.

\begin{prop}\label{curveconstruct}
   For $d\geq2$ and $C_{\Gamma} \subset \bbP^d$, there exists an irreducible, reduced curve $C \subset \bbP^d$ with a unique singular point $p$ such that the completed tangent cone $\CTC_pC$ contains $C_{\Gamma}$.
\end{prop}

\begin{proof}
    {Let $\pi:Y=\operatorname{Bl}_p\bbP^d\to \bbP^d$ be the blowup of $\bbP^d$ at the vertex $p$, with exceptional divisor $E\simeq \bbP^{d-1}$. The points of $\Gamma$ determine points of $E$, namely the tangent directions of the corresponding lines through $p$. Let $\mu:\widehat Y=\operatorname{Bl}_{\Gamma}(Y)\to Y$ be the blowup of the reduced finite set $\Gamma$, with exceptional divisor $F=\sum_{z\in \Gamma}F_z$. Fix an ample line bundle $A$ on $Y$. Since $\mathcal O_{\widehat Y}(-F)$ is $\mu$-very ample, for $\ell \gg0$ the line bundle $M=\mu^*A^{\otimes \ell}\otimes\mathcal O_{\widehat Y}(-F)$ is very ample. Choose $d-1$ general divisors in $|M|$, and let $C'\subset \widehat Y$ be their intersection. By repeated applications of Bertini's Theorem, the curve $C'$ is smooth and irreducible. Since $M|_{F_z}\cong \mathcal O_{\bbP^{d-1}}(1)$, these divisors restrict to $d-1$ general hyperplanes on each $F_z$; hence $C'$ meets each $F_z$ transversely in one reduced point. This transversality implies that $\mu$ maps $C'$ isomorphically near $C'\cap F_z$ to a smooth curve through $z$. Let $\widetilde C=\mu(C')$. Then $\widetilde C\subset Y$ is a smooth irreducible curve passing through every point of $\Gamma$, with independent defining differentials at those points. A general choice also ensures that $\widetilde C\not\subset E$, so $\widetilde C\cap E$ is finite.}

    {Let $C=\pi(\widetilde C)$ with its reduced structure. Since $\pi$ is an isomorphism away from $E$, the curve $C$ is smooth away from $p$, and it is irreducible and reduced. Thus $\pi|_{\widetilde C}:\widetilde C\to C$ is the normalization map. The preimage of $p$ in $\widetilde C$ contains the points of $\Gamma\subset E$; since $\Gamma$ spans $\bbP^{d-1}$ and $d\geq2$, this gives at least two branches through $p$, so $p$ is singular. Moreover, $\widetilde C$ is the strict transform of $C$, and under the standard identification of the exceptional fiber of $\operatorname{Bl}_p C$ with $\PTC_pC$, the scheme-theoretic intersection $\widetilde C\cap E$ is the projectivized tangent cone $\PTC_pC\subset E=\bbP \operatorname T_p\bbP^d$. Since $\Gamma\subset \widetilde C\cap E$, the scheme $\PTC_pC$ contains $\Gamma$ as a reduced closed subscheme. Taking projective cones, $\CTC_pC$ contains the cone over $\Gamma$, which is $C_\Gamma$.}
\end{proof}

We are now ready to state and prove our main theorem.

\begin{theorem}\label{main-thm-count-fam}
 Let $\mcF_i \to B_i$ be a countable collection of families of varieties of dimensions strictly smaller than $d$, and suppose $\dim(B_i) \leq M$ for every $i$. Then there exists a reduced, irreducible curve $C \subset \bbP^d$ with a single singular point which does not embed in any member of any $\mcF_i$. \end{theorem}

 \begin{proof}
      
     Let $|\Gamma|\geq M+d^2+d-1$. By Corollary \ref{cor-count-fam}, the very general $C_\Gamma$ does not admit a vertex-preserving closed immersion $C_\Gamma\cong \CTC_p(C_\Gamma)\hookrightarrow \CTC_qX$ for any point $q$ of any member $X$ of any $\mcF_i$. Pick such a $C_\Gamma$.
     By Proposition \ref{curveconstruct}, there exists a reduced, irreducible curve $C \subset \bbP^d$ with a unique singular point $p$ such that $\CTC_pC$ contains $C_\Gamma$.
     
     Suppose there were an embedding $C\hookrightarrow X$, where $X$ is a member of some $\mcF_i$, and let $q$ be the image of $p$. Then functoriality of completed tangent cones would give a closed immersion $\CTC_pC\hookrightarrow \CTC_qX$. Since $C_\Gamma\subset \CTC_pC$, this would give a vertex-preserving closed immersion $C_\Gamma\hookrightarrow \CTC_qX$, a contradiction. We thus conclude that $C$ does not embed in any member of any of the $\mcF_i$.
 \end{proof}

Theorem \ref{main_thm} is now a corollary of Theorem \ref{main-thm-count-fam}.

\begin{proof}[Proof of \Cref{main_thm}]
There are countably many projective normal toric varieties of dimension less than $d$. Regard each of them as a family over $\operatorname{Spec} k$. Theorem \ref{main-thm-count-fam} then gives a reduced, irreducible curve $C\subset \bbP^d$ which does not embed in any projective normal toric variety of dimension less than $d$.
\end{proof}

\begin{remark}\label{rem:global-obstruction} {
 Although concurrent lines can be embedded in birational projections of
cones over rational normal curves, there are global obstructions to
embedding even smooth curves in such surfaces. Let
$R\subseteq\bbP^{d-1}$ be a smooth rational curve of degree $n$, and let
$\hat R\subseteq\bbP^d$ be its cone. As described in
Section~\ref{subsec:hirzebruch-surfaces}, there is a proper birational
morphism $\phi:\bbF_n\to\hat R$ which restricts to an isomorphism
$\bbF_n\setminus E\simeq\hat R\setminus\{p\}$, where $p$ is the cone
point. Any embedding of a smooth curve $C\hookrightarrow\hat R$ lifts
to an embedding $C\hookrightarrow\bbF_n$.

We now show more generally that a general smooth curve of genus $g>10$ does not embed in any Hirzebruch surface $\bbF_n$ for $n \geq 0$. By the Brill--Noether
Theorem \cite{GH80}, a general smooth curve $C$ of genus $g$ has
gonality $\lfloor(g+3)/2\rfloor$, which is at least
$\frac{g}{2}+1$. Write the class of $C$ as $aE+bF$, where $E$ is a section with
$E^2=-n$ and $F$ is a fiber. The ruling restricts to a map of degree
$a=C\cdot F$ on $C$, so $a\geq\frac{g}{2}+1$. Since $C$ is not rational, we
have $C\neq E$, and therefore $b-an=C\cdot E\geq0$. Using
$K_{\bbF_n}=-2E-(n+2)F$, adjunction gives
\[
\begin{aligned}
2g-2
  &=(K_{\bbF_n}+C)\cdot C\\
  &=(b-an)(a-2)+a(b-n-2).
\end{aligned}
\]
If $n\geq1$, the first summand is nonnegative, while
$b\geq an$ and $a\geq\frac{g}{2}+1$ give
\[
2g-2\geq a(b-n-2)
\geq\left(\frac{g}{2}+1\right)
       \left(\frac{g}{2}-2\right),
\]
which is impossible for $g>10$. If $n=0$, the two rulings have degrees
$a$ and $b$ on $C$, so both are at least $\frac{g}{2}+1$. In this case
adjunction gives
\[
g=(a-1)(b-1)\geq\frac{g^2}{4},
\]
which is already impossible for $g>4$. We conclude that a general
smooth curve of genus $g>10$ does not embed in any $\bbF_n$ and hence cannot
embed in $\hat R$.}

\end{remark}

\subsection*{Further questions} Our results naturally lead to several questions.

\begin{question}
Which reduced and irreducible curves admit embeddings into a normal projective toric threefold? Which admit embeddings into a weighted projective threefold?
\end{question}

It is also natural to consider higher-dimensional analogues of the embedding problems studied in this paper. By taking an $e$-dimensional complete intersection in $\bbP^d$
containing a curve $C_\Gamma$ that cannot embed in any projective normal toric
variety of dimension less than $d$, we obtain an $e$-dimensional
variety with the same property.

However, in contrast with the case of curves, the corresponding embedding problem for higher-dimensional smooth varieties is substantially more subtle. For instance, in characteristic zero a smooth projective rationally connected threefold contains every smooth projective curve, whereas analogous embedding statements for smooth projective surfaces already fail in simple toric examples.

\begin{example}
    Every smooth projective surface admits an embedding into $\bbP^5$, but not every smooth projective surface embeds into $(\bbP^1)^5$. In fact, there are smooth surfaces that do not embed in $(\bbP^1)^n$ for any $n$. Suppose that $S$ is a smooth surface with \[\operatorname{Pic}(S)\cong \mathbb Z H,\]
where $H$ is ample. For
example, one may take $S=\bbP^2$. By the Noether--Lefschetz Theorem, in
characteristic zero a very general smooth surface of degree $d\geq4$
in $\bbP^3$ also has this property.

Assume that $S$ admits an embedding into $(\bbP^1)^n$. Let
\[\pi_i: S \to\bbP^1,\qquad 1\le i\le n,\]
denote the compositions with the coordinate projections. If $\pi_i$ is
nonconstant, then the line bundle
$L_i=\pi_i^*\mathcal O_{\bbP^1}(1)$ is nef and nontrivial, and therefore
has numerical class $aH$ for some integer $a>0$. Since $H$ is ample,
$L_i^2=a^2H^2>0$. On the other hand, $L_i^2=0$ because $L_i$ is pulled
back from $\bbP^1$, a contradiction. Thus every projection $\pi_i$ is
constant, contradicting the assumption that $S$ is embedded in
$(\bbP^1)^n$. 
\end{example}

This example suggests the following problems.

\begin{question}
Which normal projective toric varieties admit embeddings of every smooth projective surface?
\end{question}

\begin{question}
Given a weighted projective fivefold, which smooth projective surfaces
embed in it? More generally, given a normal projective toric fivefold,
which reduced and irreducible projective surfaces embed in it?
\end{question}

\end{document}